\documentclass{amsart}
\usepackage{amsthm,amssymb,tikz,enumitem,color,hyperref}
\usetikzlibrary{arrows,matrix}
\newcommand{\linfty}{\ensuremath{\ell^\infty}}
\newcommand{\Linfty}{\ensuremath{L^\infty}}
\newcommand{\cat}[1]{\ensuremath{\mathbf{#1}}}
\newcommand{\Cstar}{\cat{Cstar}}

\newcommand{\AWstar}{\cat{AWstar}}
\newcommand{\Z}{\mathbb{Z}}
\newcommand{\C}{\mathbb{C}}
\newcommand{\T}{\mathbb{T}}
\newcommand{\N}{\mathbb{N}}
\newcommand{\inprod}[2]{\ensuremath{\langle #1 , #2 \rangle}}
\newcommand{\eps}{\varepsilon}

\theoremstyle{plain}

\newtheorem*{theorem*}{Theorem}
\newtheorem*{corollary*}{Corollary}
\newtheorem{lemma}{Lemma}
\theoremstyle{definition}
\newtheorem*{definition*}{Definition}
\newtheorem*{remark*}{Remark}
\newtheorem*{notation*}{Notation}

\begin{document}
\title{On discretization of C*-algebras}
\author{Chris Heunen}
\address{Department of Computer Science, University of Oxford, Oxford OX1 3QD, UK}
\email{heunen@cs.ox.ac.uk}
\author{Manuel L. Reyes}
\address{Department of Mathematics, Bowdoin College\\
Brunswick, ME 04011--8486, USA}
\email{reyes@bowdoin.edu}
\thanks{C.\ Heunen was supported by EPSRC Fellowship EP/L002388/1. \\\indent M.\, L.\ Reyes was supported by NSF grant DMS-1407152}
\date{March 25, 2015}
\maketitle
\begin{abstract}
  The C*-algebra of bounded operators on the separable infinite-dimensional Hilbert space cannot be mapped to an 
  AW*-algebra in such a way that each unital commutative C*-subalgebra $C(X)$ factors normally
  through $\linfty(X)$. Consequently, there is no faithful functor discretizing C*-algebras to 
  AW*-algebras, including von Neumann algebras, in this way.
\end{abstract}

\section{Introduction}

In operator algebra it is common practice to think of a C*-algebra as representing a noncommutative 
analogue of a topological space, and to think of a W*-algebra as representing a noncommutative
analogue of a measurable space. 
What would it mean to make precise the notion of a C*-algebra $A$ as a `noncommutative ring of
continuous functions'? The present article explores the idea that one should first embed $A$ in an
appropriate noncommutative algebra of `bounded functions on the underlying quantum set of the
spectrum of $A$', just like any topological space embeds in a discrete one~\cite{akemann, gileskummer}.
It is tempting to demand that such a `noncommutative function ring' be an atomic W*-algebra, but
we work more generally under the mere assumption that they be AW*-algebras. 

Write $\Cstar$ for the category of unital C*-algebras with unital $*$-homomorphisms, and
$\AWstar$ for the category of AW*-algebras with unital $*$-homomorphisms
whose restriction to the projection lattices preserve arbitrary least upper bounds.\footnote{See~\cite[Lemma~2.2]{heunenreyes:activelattice} for further characterizations of these 
morphisms.} The discussion above leads naturally to the following notion, in keeping with the programme 
of taking commutative subalgebras seriously~\cite{heunen:faces,reyes:obstructing,bergheunen:extending,reyes:sheaves,vdbergheunen:colim}, that has recently been successful~\cite{heunenlandsmanspitters:topos,hamhalter:ordered,heunenreyes:activelattice,hamhalter:dye}.

\begin{definition*}
	A \emph{discretization} of a unital C*-algebra $A$ is a unital $*$-homomorphism $\phi \colon A \to M$ to an AW*-algebra $M$ whose restriction to each commutative unital C*-subalgebra $C \cong C(X)$ factors through the natural inclusion $C(X) \to \linfty(X)$ via a morphism $\linfty(X) \dashrightarrow M$ in $\AWstar$, so that the following diagram commutes.
	\[\begin{tikzpicture}[xscale=3,yscale=1]
	  \node (tl) at (0,1) {$A$};
	  \node (tr) at (1,1) {$M$};
	  \node (bl) at (0,0) {$C(X)$};
	  \node (br) at (1,0) {$\linfty(X)$};
	  \draw[->] (tl) to node[above] {$\phi$} (tr);
	  \draw[right hook->] (bl) to (br);
	  \draw[left hook->] (bl) to (tl);
	  \draw[->,dashed] (br) to (tr); 
	\end{tikzpicture}\]
\end{definition*}

This short note proves that this construction degenerates in prototypical cases.

\begin{theorem*}
  If $\phi \colon B(H) \to M$ is a discretization for a separable infinite-dimensional Hilbert space $H$, then $M=0$.
\end{theorem*}

For W*-algebras $M$, this obstruction concretely means that 
$B(H)$ has no nontrivial representation on a Hilbert space such that every (maximal) commutative
$*$-subalgebra has a basis of simultaneous eigenvectors.

Consequently, discretization cannot be made into a faithful functor. 

\begin{corollary*}
  Let a functor $F \colon \Cstar \to \AWstar$ have natural unital $*$-homo\-morphisms
   $\eta_A \colon A \to F(A)$.
  Suppose there are isomorphisms $F(C(X)) \cong \linfty(X)$ for each compact Hausdorff space $X$ that
  turn $\eta_{C(X)}$ into the inclusion $C(X) \to \linfty(X)$.
  If a unital C*-algebra $A$ has a unital $*$-homomorphism $\alpha \colon B(K) \to A$ for an 
  infinite-dimensional Hilbert space $K$, then $F(A)=0$.
\end{corollary*}

As the proof of the Theorem relies on the use of annihilating projections and on the Archimedean
property of the partial ordering of positive elements in the discretizing AW*-algebra $M$, 
it is intriguing to note that this does not rule out faithful functors $F$ as above from $\Cstar$ to the
category $\Cstar$ or to the category of Baer $*$-rings with 
$*$-homomorphisms that restrict to complete orthomorphisms on projection lattices.
A rather different approach to the problem of extending the embeddings $C(X) \hookrightarrow \linfty(X)$
to noncommutative C*-algebras has recently appeared in~\cite{kornell:vstar}.
We also remark that since the identity functor discretizes all finite-dimensional C*-algebras, this truly
infinite-dimensional obstruction is independent of the Kochen--Specker theorem, a key ingredient
in some previous spectral obstruction results~\cite{reyes:obstructing, bergheunen:extending}.

The rest of this note proves the Theorem and its Corollary.

\section{Proof}

\begin{notation*}
  Fix a separable infinite-dimensional Hilbert space $H=L^2[0,1]$, and consider its algebra $B(H)$ of bounded operators. Write $D$ for the discrete maximal abelian \mbox{$*$-subalgebra} generated as a W*-algebra by the projections $q_n$ onto the Fourier basis vectors $e_n = \exp(2\pi i n-)$ for $n \in \Z$. There is a canonical conditional expectation $E \colon B(H) \to D$ that sends $f \in B(H)$ to its diagonal part $\sum q_n f q_n$. 
\end{notation*}

The main results rely upon the following mild strengthening of the recent solution of the Kadison--Singer problem~\cite{marcusspielmansrivastava:kadisonsinger}. 

\begin{lemma}\label{lem:kadisonsinger}
  Let $A$ be any unital C*-algebra, and $\psi_0 \colon D \to \C$ a pure state of $D$.
  The map $\psi_0 \cdot 1_A \colon D \to A$ given by $f \mapsto \psi_0(f) \cdot 1_A$ extends uniquely to a unital completely positive map $\psi \colon B(H) \to A$ given by $f \mapsto \psi_0(E(f)) \cdot 1_A$.
\end{lemma}
\begin{proof}
  We employ a standard reduction of the unique extension problem to Anderson's paving conjecture, as outlined, for instance, in~\cite{paulsenraghupathi:paving}.

  The extension $\psi_0 \circ E$ is well known to be a pure state, proving existence. For uniqueness, let $\psi \colon B(H) \to A$ be any unital completely positive map extending $\psi_0 \cdot 1_A$.
  It suffices to show that $\psi = \psi \circ E$, as then $E(f) \in D$ for $f \in B(H)$ implies $\psi(f) = \psi(E(f)) = \psi_0(E(f)) \cdot 1_A$ as desired.
  As $f$ is a linear combination of two self-adjoint elements, we may further assume that $f = f^* \in B(H)$. 
  Replacing $f$ with $f - E(f)$, we reduce to showing $\psi(f)=0$ when $f = f^*$ and $E(f) = 0$.
  To this end, let $\eps > 0$. 
  By Anderson's paving conjecture, established in~\cite[1.3]{marcusspielmansrivastava:kadisonsinger}, there exist projections $p_1, \dots, p_n \in D$ with $\sum p_i = 1$ and $\|p_i f p_i\| \leq \eps \|g\|$ for all $i$.
  As $\psi|_D = \psi_0$ is a pure state, up to reordering indices we have $\psi(p_1) = 1$ and 
  $\psi(p_i) = 0$ for $i > 1$.

  By the Schwarz inequality for 2-positive maps~\cite[Exercise~3.4]{paulsen:completelybounded}, for all 
  $i > 1$ we have $\|\psi(p_i f) \|^2 \leq \|\psi(p_i p_i^*)\| \cdot \| \psi(f^* f)\| = 0$ since $\psi(p_i p_i^*) 
  = \psi(p_i) = 0$. Thus $\psi(p_i f) = 0$ for all $i > 1$, making $\psi(f) = \sum_{i=1}^n \psi(p_i f) = \psi(p_1 f)$. 
  A symmetric argument replacing $f$ with $p_1 f$ yields $\psi(f) =  \psi(p_1f) = \psi(p_1 f p_1)$. 
  Unitality of $\psi$ furthermore gives $\|\psi\|=1$~\cite[Corollary~2.8]{paulsen:completelybounded}, so that
  \[
    \|\psi(f)\| = \|\psi(p_1 f p_1)\| \leq \| p_1 f p_1 \| \leq \eps \|f\|.
  \]
  As $\eps$ was arbitrary, we deduce that $\psi(f) = 0$ as desired.
\end{proof}

Note that Lemma~\ref{lem:kadisonsinger} still holds with $\psi$ merely 2-positive.
Next we consider the continuous maximal abelian *-subalgebra $C=\Linfty[0,1]$ of $B(H)$.

\begin{lemma}\label{lem:fourier}
  Let $\psi \colon B(H) \to \mathbb{C}$ be the unique extension of a pure state of $D$.
  The restriction of $\psi$ to $C$ is the state given by integration (against the Lebesgue measure).
\end{lemma}
\begin{proof}
  Each $f \in C$ has diagonal part $E(f) = \int_0^1 f(x)\, \mathrm{d}x$ because
  \begin{align*}
    \inprod{f e_n}{e_n} 
    & = \inprod{f \cdot \exp(2 \pi in-)}{\exp(2 \pi in-)} \\
    & = \int_0^1 f(x) \cdot e^{2 \pi inx} \cdot \overline{e^{2 \pi inx}} \, \mathrm{d}x \\
    & = \int_0^1 f(x) \, \mathrm{d}x.
  \end{align*}
  Because we assumed that $\psi$ is a pure state of $D$, we have $\psi = \psi \circ E$ as in Lemma~\ref{lem:kadisonsinger}.
  Hence $\psi(f) = \psi(E(f)) = \psi( \int_0^1 f(x)\,\mathrm{d}x) = \int_0^1 f(x)\,\mathrm{d}x$.
\end{proof}

To prove the Theorem, recall that for an orthogonal set of projections $\{p_i\}$ in an AW*-algebra, 
$\sum p_i$ denotes their least upper bound in the lattice of projections.

\begin{proof}[Proof of Theorem]
  Write $C \cong C(X)$ and $D \cong C(Y)$ for compact Hausdorff spaces $X$ and $Y$. The
  discretization $\phi \colon B(H) \to M$ is accompanied by the following commutative
  diagram, where $\alpha$ and $\beta$ are morphisms in $\AWstar$.
  \[\begin{tikzpicture}[xscale=3]
    \node at (-.56,1) {$C = \Linfty[0,1] \cong$};
    \node at (-.52,-1) {$D = \linfty(\Z) \cong$};
    \node (tl) at (0,1) {$C(X)$};
    \node (l) at (0,0) {$B(H)$};
    \node (bl) at (0,-1) {$C(Y)$};
    \node (tr) at (1,1) {$\linfty(X)$};
    \node (r) at (1,0) {$M$};
    \node (br) at (1,-1) {$\linfty(Y)$};
    \draw[->] (tl) to (tr);
    \draw[->] (bl) to (br);
    \draw[->] (l) to node[above] {$\phi$} (r);
    \draw[right hook->] (tl) to (l);
    \draw[left hook->] (bl) to (l);
    \draw[->, dashed] (tr) to node[right] {$\alpha$} (r);
    \draw[->, dashed] (br) to node[right] {$\beta$} (r);
  \end{tikzpicture}\]
  The atomic projections $\delta_x \in \linfty(X)$ for $x \in X$ and $\delta_y \in \linfty(Y)$ for 
  $y \in Y$ have respective images $p_x = \alpha(\delta_x) \in M$ and $q_y = \beta(\delta_y) \in M$.
  For each $y$, the map $\psi \colon B(H) \to q_y M q_y$ given by $\psi(f) = q_y \phi(f) q_y$ is
  completely positive and unital (where $q_y$ is the unit of $q_y M q_y$). Its restriction to 
  $D$ is of the following form, where we consider $f \in D$ as an element of
  the function algebra $C(Y) \subseteq \linfty(Y)$: 
  \[
  \psi(f) = q_y \phi(f) q_y = \beta(\delta_y f \delta_y) = \beta(f(y) \delta_y) = f(y) q_y.
  \]
  Thus there is a pure state $\psi_0$ on $D$ with $\psi|_D = \psi_0 \cdot q_y$. It follows from
  Lemma~\ref{lem:kadisonsinger} that $\psi = (\psi_0 \circ E) \cdot q_y$.
  For $t \in [0,1]$, write $e_t = \phi(\chi_{[0,t]})$ for the image of the characteristic function 
  $\chi_{[0,t]} \in C$. Lemma~\ref{lem:fourier} implies $\psi(\chi_{[0,t]}) = \left( \int_0^1 \chi_{[0,t]}(x)\,\mathrm{d}x \right) \cdot q_y = t q_y$, so
  \[
    q_y e_t q_y = q_y \phi(\chi_{[0,t]}) q_y = \psi(\chi_{[0,t]}) = t q_y
  \] 
  for all $y \in Y$ and all $t \in [0,1]$.

  Considering each projection $\chi_{[0,t]}$ as an element of $C(X)$, fix clopen sets $K_t \subseteq X$
  such that $\chi_{[0,t]} = \sum_{x \in K_t} \delta_x$.
  Then $e_t = \phi(\chi_{[0,t]}) = \sum_{x \in K_t} p_x$ in $M$. Fix $n \in \N$, and set 
  $J_i = K_{i/n} \setminus K_{(i-1)/n} \subseteq Y$. Note that $K_1 = X$, so that these $J_i$ partition 
  $X$ into a
  disjoint union of $n$ clopen sets. By construction, 
  \[
    \sum_{x \in J_i} p_x 
    \;=\; \sum_{x \in K_{i/n}} p_x - \sum_{x \in K_{(i-1)/n}} p_x \\
    \;=\; e_{i/n} - e_{(i-1)/n}.
  \]
  Now fix $x \in X$. Then $x \in J_i$ for some $i$, and $p_x \leq e_{i/n} - e_{(i-1)/n}$ as above.
  Thus
  \[
    q_y p_x q_y 
    \;\leq\; q_y (e_{i/n} - e_{(i-1)/n}) q_y \\
    \;=\; \tfrac{i}{n} \, q_y - \tfrac{i-1}{n} \, q_y \\
    \;=\; \tfrac{1}{n} \, q_y.
  \]
  As $n$ was arbitrary, we find that $q_y p_x q_y = 0$.
  Now $(p_x q_y)^* (p_x q_y) = q_y p_x q_y = 0$ gives $p_x q_y = 0$ for all $y \in Y$. Thus $p_x$ is
  orthogonal to $\sum q_y = 1$ in $M$, whence $p_x = 0$ for all $x \in X$. It follows that 
  $1 = \sum p_x = 0$ in $M$, and so $M = 0$.
\end{proof}

\begin{remark*}
We thank an anonymous referee for noticing that our arguments prevail without the full force of Kadison--Singer. This may be done as follows. Identifying the algebra 
$C(\T)$ of continuous functions on the unit circle $\T$ with the subalgebra $\{f \mid f(0) = f(1)\}
\subseteq C[0,1]$, it is known that $C(\T)$ satisfies paving with respect to $D$.
(Indeed, the algebra of Fourier polynomials---or more generally, the Wiener algebra $A(\T)$---is a 
dense subalgebra of $C(\T)$ and lies in the algebra $M_0 \subseteq B(H)$ of operators that
are $l_1$-bounded in the sense of Tanbay~\cite{tanbay:extensions} with respect to the Fourier basis 
$\{e_n \mid n \in \Z\}$. Thus $C(\T)$ lies in the norm closure $M$ of $M_0$, 
and~\cite{tanbay:extensions} shows that all operators in $M$ can be paved with respect to $D$.) 
An argument as in Lemma~\ref{lem:kadisonsinger} shows that the completely positive map $\psi$
in the proof of the Theorem is uniquely determined on $C(\T)$, and a computation as in
Lemma~\ref{lem:fourier} shows that this extension is the state corresponding to the arclength measure
on $\T$. The Theorem may now be proved in essentially the same manner, replacing
 $C$ with $C(\T)$.
\end{remark*}

The proof of the Corollary uses stability of discretizations in the following sense.

\begin{lemma}\label{lem:stable}
If $\phi \colon B \to M$ is a discretization, $\alpha \colon A \to B$ is a morphism in $\Cstar$, and $\beta \colon M \to N$ is a morphism in $\AWstar$, then $\beta \circ \phi \circ \alpha$ discretizes $A$. 
\end{lemma}
\begin{proof}
  If $C(X) \subseteq A$ is a commutative C*-subalgebra, so is $C(Y) \cong \alpha[C(X)] \subseteq B$, making the top squares of the following diagram commute (where $\widehat{\alpha} \colon Y \to X$ is the  continuous function corresponding to $\alpha$ via Gelfand duality).
  \[\begin{tikzpicture}[xscale=3,yscale=1]
    \node (A) at (0,1) {$A$};
    \node (B) at (1,1) {$B$};
    \node (M) at (2,1) {$M$};
    \node (N) at (3,1) {$N$};
    \node (C) at (0,0) {$C(X)$};
    \node (D) at (1,0) {$C(Y)$};
    \node (l1) at (2,0) {$\linfty(Y)$};
    \node (l2) at (1,-1) {$\linfty(X)$};
    \draw[->] (A) to node[above] {$\alpha$} (B);
    \draw[->] (B) to node[above] {$\phi$} (M);
    \draw[->, dashed] (M) to node[above] {$\beta$} (N);
    \draw[left hook->] (C) to (A);
    \draw[left hook->] (D) to (B);
    \draw[->] (C) to node[above] {$C(\widehat{\alpha})$} (D);
    \draw[right hook->] (D) to node[above] {$\eta_{C(Y)}$} (l1);
    \draw[->, dashed] (l1) to (M);
    \draw[right hook->] (C.-60) to node[below=1mm] {$\eta_{C(X)}$} (l2);
    \draw[->, dashed] (l2) to node[below=1mm] {$\linfty(\widehat{\alpha})$} (l1);
  \end{tikzpicture}\]
  The bottom triangle commutes by naturality of $\eta$. 
  As all dashed arrows are morphisms in $\AWstar$, so is their composite.
\end{proof}

\begin{proof}[Proof of Corollary]
  Let $\gamma \colon C(X) \rightarrowtail A$ be the embedding of a commutative C*-subalgebra.
  The hypotheses ensure that the following diagram commutes, where $F(\gamma)$ is a morphism in $\AWstar$, making $\eta_A \colon A \to F(A)$ a discretization.
  \[\begin{tikzpicture}[xscale=3,yscale=1.25]
    \node (A) at (0,1) {$A$};
    \node (FA) at (1.5,1) {$F(A)$};
    \node (C) at (0,0) {$C(X)$};
    \node (l) at (1,0) {$\linfty(X) \cong$};
    \node(FC) at (1.5,0) {$F(C(X))$};
    \draw[->] (A) to node[above] {$\eta_A$} (FA);
    \draw[->] (FC) to node[right] {$F(\gamma)$} (FA);
    \draw[>->] (C) to node[left] {$\gamma$} (A);
    \draw[right hook->] (C) to (l);
  \end{tikzpicture}\]

  Since $K$ is infinite-dimensional, it is unitarily isomorphic to $H \otimes K$, so $a \mapsto a \otimes 1$ is 
  a unital $*$-homomorphism $\iota \colon B(H) \to B(H) \otimes B(K) \cong B(K)$.
  Lemma~\ref{lem:stable} implies $\eta_A \circ \alpha \circ \iota \colon B(H) \to F(A)$ is a 
  discretization, and the Theorem gives $F(A)=0$.
\end{proof}

\bibliographystyle{plain}
\bibliography{nodiscretization-arxiv-v4}

\end{document}